\documentclass[11pt,a4paper]{article}
\usepackage{latexsym} 
\usepackage{amsmath,amssymb,amsfonts,amsthm}
\usepackage{verbatim}
\usepackage[dvips]{graphicx}
\usepackage{textcomp}

\newtheorem{theorem}{Theorem}
\newtheorem{proposition}{Proposition}
\newtheorem{corollary}{Corollary}
\newtheorem{definition}{Definition}
\newtheorem{lemma}{Lemma}

\newcommand{\mP }{\mathbb{P}}
\newcommand{\me}{\mathbb{E}}
\newcommand{\mn}{\mathbb{N}}
\newcommand{\mr }{ \mathbf{R}}
\newcommand{\dive}{\textrm{div}_E }
\newcommand{\grade}{\textrm{grad}_E }
\newcommand{\grad}{\textrm{grad}}

\newcommand{\supp}{\textrm{supp} }
\newcommand{\F}{\mathcal{F}}
\newcommand{\bbl}{\mathcal{B}^L_b}

\begin{document}
\begin{center}


 {\Large {\bf Harmonic measures in embedded foliated manifolds  }}

\end{center}

\vspace{0.1cm}

\begin{center}
{\large { Pedro J. Catuogno}\footnote{E-mail:
pedrojc@ime.unicamp.br. 
Research partially supported by CNPq 302.704/2008-6, 480.271/2009-7
and FAPESP 07/06896-5.} \ \ \ \ \ \ \ \ \ \ \ \ \  Diego S.
Ledesma\footnote{E-mail:  Research
supported by CNPQ, grant no. 142655/2005-8}

\bigskip

{ Paulo R. Ruffino}\footnote{Corresponding author, e-mail:
ruffino@ime.unicamp.br.
Research partially supported by CNPq 306.264/2009-9, 480.271/2009-7
and FAPESP 07/06896-5.}}

\vspace{0.2cm}

\textit{Departamento de Matem\'{a}tica, Universidade Estadual de Campinas, \\
13.083-859- Campinas - SP, Brazil.}

\end{center}

\begin{abstract}

We study harmonic and totally invariant measures in a foliated compact Riemannian
 manifold isometrically embedded in an Euclidean space. We introduce geometrical techniques 
for stochastic calculus
in this space. In particular, using these
techniques we can construct explicitely an Stratonovich equation for the
foliated Brownian motion (cf. L. Garnett \cite{LG} and others). We present a
characterization of totally invariant measures in terms of the flow of
diffeomorphisms of associated to this equation. We prove an ergodic formula
for the sum of the Lyapunov exponents in terms of the geometry of the leaves.


\end{abstract}


\vspace{0.3cm}
\noindent {\bf MSC2010 subject classification:} 58J65, 53C12 (60H30, 60J60).

\vspace{8mm}

\section{Introduction}

\bigskip

The main topic of this article is to study harmonic and totally invariant measures 
in a foliated compact Riemannian manifold $M$ isometrically embedded in
an Euclidean space. Our technique offers some tools for the geometrical analysis
of stochastic processes in $M$, in particular, it allows one to
construct explicitely an Stratonovich equation for the foliated Brownian motion
(as introduced in the literature by L. Garnett \cite{LG}, see also
\cite{candel}, \cite{candel2} and
references therein).

Next section introduces the geometrical background, in particular we present the
tension $\kappa$ in the tangent of the leaves which measures the difference
between the divergent operator in $M$ and the divergent operator in the leaves.
Among others interesting properties, we prove that  $\kappa$ is related to
the Gobillon-Vey class of the foliated space (Proposition  \ref{prop-Gb-V}). In
Section 3 we put these geometrical tools to play with invariant measures:
harmonic, totally invariant and holonomy invariant measures (when they exist).
We present a characterization of totally invariant measures in terms of the flow
of diffeomorphisms of our foliated Brownian system (Theorem
\ref{thm_carac_mti}), in the proof we use currents. Ergodic properties appears in the last section, where the
main result is an ergodic formula for the sum of the Lyapunov exponents in terms
of the geometry of the leaves: the tensor $\kappa$ and the mean curvature $H$
(Theorem \ref{mean lyap exp fobm}).


\section{Foliated Geometry}\label{folgeo}

We fix $(M, \mathcal{F}, g)$ a foliated Riemannian manifold without boundary.
The foliation $\mathcal{F}$ is
given by the integrable subbundle $E$ of tangent vectors to $\mathcal{F}$.

We are interested in describing the foliated operators over $(M,\mathcal{F},g)$
in terms of an isometric embedding of the Riemannian manifold $(M,g)$ in an
Euclidean space $\mr^N$ as guaranteed by the classical Nash theorem. 
Let $P:T\mr^N|_M\rightarrow TM$ be defined by $P(m,v)=P(m)v, $ where
$P(m):\mr^N\rightarrow T_mM$ is the orthogonal projection.
So, the Riemannian connection $\nabla$ of $M$ can be written as
\[
\nabla_VY=PdY(V),
\]
for all sections $V$ and $Y$ in $TM$. 

The connection $\nabla^E$ is defined on $E$ in terms of the orthogonal
projection $\pi:TM\rightarrow E$ by
\[
\nabla^E_VY=\pi \nabla_VY 
\]
for all sections  $V\in TM$ and $Y\in E$.

The elementary differential operators we are going to deal with in the calculus
in a foliated
manifolds are the following (cf. \cite{NA}):
\begin{definition}\label{z2}
Let $f$ be a smooth function and $X,\: Y$ sections of $E$, we define the
operators

\noindent \textbf{a)} foliated gradient $\grade\:f=\pi (\grad~ 
f)$;

\noindent \textbf{b)} foliated divergence
$\textrm{div}_E\:Y=Tr_E~g(\nabla_\cdot^EY,\cdot),$ where
$Tr_E$ is the trace on $E$;

\noindent \textbf{c)} foliated Hessian 
$\textrm{Hess}_E(f)(X,Y)=XY(f)-\nabla^E_XYf$;

\noindent \textbf{d)} foliated Laplacian
$\Delta_Ef=div_E(grad_E\:f) $.
\end{definition}
\noindent The restriction to a leaf of the
operators $\grade$, $\textrm{div}_E$ and $\Delta_E$ are the
corresponding operators on the leaf with the induced metric. The following lemma
is a natural consequence of the isometric embedding of $M$
into $\mr^N$. In particular it gives a 
description of the foliated Laplacian $\Delta_E$ as a sum of squares of vector
fields over $M$. This will be useful to get the foliated Brownian motion as a
solution of a Stratonovich stochastic differential equation.

Let $\{e_i:~i=1,\ldots ,N\}$ be an orthonormal basis  of $\mr^N$. We denote by
$\tilde X_i$ the gradient vector field $Pe_i$\
and by $X_i$ the foliated gradient vector fields $\pi\tilde X_i$.

\begin{lemma}\label{ab}Let $f$ be a smooth function and $X$ a section of $E$.
Then 
\begin{itemize}
\item[a)] $\grade\:f=\sum_{i=1}^N X_if\:X_i,$
\item[b)] div$_E(X)=\sum_{i=1}^Ng(\nabla^E_{X_i}X,X_i),$
\item[c)] $\Delta_Ef=\sum_{i=1}^N X_i^2f$.
\end{itemize}
\end{lemma}
\begin{proof}
 We first observe that
\begin{eqnarray}
\sum_{i=1}^NX_i(m)\otimes X_i(m)=\sum_{i=1}^pu_i\otimes u_i,\label{ap1}
\end{eqnarray}
where the $\{u_i\}$ is an orthonormal basis of $E_m$.

Item (a) follows immediately from definition and the contraction of Equation
(\ref{ap1}) with $df$.

For item (b) note that 
\begin{eqnarray*}
\textrm{div}_E(Y)&=&g(\nabla_\cdot^EY,\cdot)\left(\sum_{i=1}^pu_i\otimes
u_i\right)\\
&=&g(\nabla_\cdot^EY,\cdot)\left(\sum_{i=1}^NX_i(m)\otimes X_i(m)\right)\\
&=&\sum_{i=1}^Ng(\nabla^E_{X_i}X,X_i).
\end{eqnarray*}

For item (c): By equation  (\ref{ap1}), we have that
\begin{eqnarray*}
\Delta_Ef&=&\sum_{i=1}^Ng(\nabla^E_{X_i}\grade f,X_i)\\
&=&\sum_{i=1}^NX_ig(\grade f,X_i)-g(\grade f,\sum_{i=1}^N\nabla^E_{X_i}X_i).
\end{eqnarray*}
Therefore we just need to prove that $\sum_{i=1}^N\nabla^E_{X_i}X_i=0$. In order
to do this we
define the projectors
\[
 \tilde P=\pi\circ P\hspace{1cm} \tilde Q=I_{\mr^N}-\tilde P.
\]
Then $\tilde P\circ\tilde Q=\tilde Q\circ\tilde P=0$. Using that
\[\tilde P\circ d\tilde P=-\tilde P\circ d\tilde Q= d\tilde P \circ\tilde Q,\]
we obtain
\begin{eqnarray*}
\sum_{i=1}^N\nabla^E_{X_i}X_i&=&\sum_{i=1}^N \pi Pd\tilde P(X_i)(e_i)\\
&=&\sum_{i=1}^N \tilde P d\tilde P(\tilde P(e_i))(e_i)\\
&=&\sum_{i=1}^N d\tilde P(\tilde P(e_i))\tilde Qe_i.
\end{eqnarray*}
By the invariance of this expression with respect to the vectors $e_i$, 
we have that $d\tilde P(\tilde P(e_i))\tilde Q e_i=0$, showing the result.
 
\end{proof}

%

\begin{definition}
We define the tension $\kappa$ as the unique section of $E$ such that
\[
g(\kappa,X)=\textrm{div}_E(X)-\textrm{div}(X).
\]
for all section $X$ in $E$. 
\end{definition}


  Suppose that there exists a 1-form $\omega$ in $M $,
$||\omega||=1$ determining a
transversaly oriented codimension 1 foliation by $E=\textrm{Ker}(\omega)$, i.e.
form $\omega$ satisfies $\omega \wedge d \omega =0$.  
The integrability of $E$ guarantees the existence of a 1-form $\alpha$ such that
$d\omega=\alpha\wedge\omega$. 
The  1-form $\alpha$ determines the Godbillon-Vey class 
of the foliation by $\textrm{gv}(\omega)=[\alpha\wedge d\alpha]\in H^3_{dR}(M)$,
see Godbillon and Vey \cite{godbillon}, Moerdijk and Mr\v{c}un \cite{moerdijk} or Walczak \cite{walczak}.
We are going to prove that the tension $\kappa$ is related to the
Godbillon-Vey class:

\begin{proposition} With the notation above, \label{prop-Gb-V}
\[
 \textrm{gv}(\omega)=[\kappa^\flat\wedge d\kappa^\flat]. 
\]
\end{proposition}

\begin{proof}
Denote by $\eta=\omega^\sharp$ the nowhere
vanishing vector field associated to $\omega$. The following traces vanish
\[ 
 \sum_{i=1}^N g(\tilde X_i,\eta )g(\pi(\nabla_{\eta} X),\tilde X_i)=0\\
\]
and 
\[
\sum_{i=1}^N g(\tilde X_i, \eta)g(\nabla_{\pi\tilde X_i}X,\eta)=0.
\]
Writing each gradient vector field $\tilde X_i$ as
\[
 \tilde X_i=X_i+\eta_i,
\]
where $\eta_i=g(\tilde X_i,\eta) \eta$, then, for any $X\in \Gamma(E) $
\begin{eqnarray*}
 \textrm{div}(X)&=&\sum_{i=1}^N\left(g(\nabla_{X_i}X,X_i)+g(\nabla_{X_i}X,
N_i)+g(\nabla_{N_i}X,X_i)+g(\nabla_{N_i}X,N_i)\right)\\
&=&\textrm{div}_E(X)+||N||^2g(\nabla_NX,N).
\end{eqnarray*}

Thus
\[
 \textrm{div}(X)-\textrm{div}_E(X)=g(X,-\nabla_NN)
\]
so, $\kappa=-\nabla_NN$.
On the other hand, since $d\omega(X,Y)=0$ for all $X,Y\in E$ and
$d\omega(N,N)=0$ then
\begin{eqnarray*}
  d\omega(X,N)&=&-\omega([X,N])\\
&=&g(N,\nabla_NX)\\
&=&g(\kappa,X).
\end{eqnarray*}
Since $d\omega(X,N)=\alpha(X)$ we get that $\alpha=\kappa^\flat$.
\end{proof}


We denote by $\nu(E)$ the normal bundle of $E$ with respect to $\mr^N$, that
is
\[
 \nu(E)=\{(x,v),~x\in M, ~v\in\mr^N, ~\textrm{such that}~ v\perp E_x\}.
\] 
\begin{definition}
The second fundamental form $\alpha\in \Gamma(E^*\otimes E^*\otimes \nu(E))$ is
the unique $\nu(E)$-valued bilinear form satisfying
\[
 <\alpha(X,Y),N>=g(\tilde PdN(X),Y)
\]
for all $X,Y\in\Gamma( E)$ and $N\in \Gamma(\nu(E))$.

The mean curvature is defined by $H=Tr_E(\alpha)$.
\end{definition}


\begin{lemma}\label{suma de divergencias nula}
For all $v\in\mr^N$ we have that
\begin{equation}\label{formula 1}
<H,v>=-\dive(\tilde P(v)).
\end{equation}
\end{lemma}
\begin{proof} We observe that
\[
 \nabla^E_vX_i=-\tilde P d(\tilde Qe_i)(v).
\]
In fact, consider the decomposition
\[
 e_i=X_i+\tilde Q e_i.
\]
 Taking the directional derivative with respect to $v$ and the projection
$\tilde
P$ to $E$ we have
\[
\tilde P d(\tilde Pe_i)(v)+\tilde Pd(\tilde Qe_i)(v)=0.
\]
Using this formula, we find that
\begin{eqnarray*}
\dive (X_i) &=&\sum_{j=1}^N g(\nabla_{X_j}^E X_i,X_j)\\
&=&-\sum_{j=1}^N g(\tilde  P d(\tilde Qe_i)(X_j),X_j)\\
&=&-\sum_{j=1}^N <\tilde Qe_i, \alpha(X_j,X_j)>\\
&=&-<H,e_i>+<H,\tilde Pe_i>.
\end{eqnarray*}
Using that $<H,\tilde P(e_i)>=0$, the result follows by linearity.

\end{proof}

\begin{corollary} \label{formulas_H_X_i} The following formulae hold:
 \begin{equation}\label{formula 2}
  \sum_{i=1}^N\dive(X_i)X_i=0
\end{equation}
and
\begin{equation}\label{formula 3}
 ||H||^2=-\sum_{i=1}^NX_i\dive(X_i).
\end{equation}
\end{corollary}
\begin{proof}

Formula (\ref{formula 2}) follows by substituting Equation (\ref{formula
1}) in 
\[
 \sum_{i=1}^N\dive(X_i)X_i=-\sum_{i=1}^N<H,\tilde Qe_i>\tilde Pe_i=0.
\]

For the proof of formula (\ref{formula 3}) we calculate
$||H||^2$ and use Lemma \ref{suma de divergencias nula}:

\begin{eqnarray*}
 ||H||^2&=&\sum_{i=1}^N<H,e_i>^2\\
&=&\sum_{i=1}^N\dive(X_i)^2\\
&=&\sum_{i,j,k=1}^Ng(\nabla_{X_j}^EX_i,X_j)g(\nabla_{X_k}^EX_i,X_k)\\
&=& \sum_{k=1}^Ng\left(\nabla_{X_k}\left(\sum_{i=1}^N\dive(X_i)
X_i\right),X_k\right)-\sum_{i=1}^NX_i(\dive(X_i))\\
&=& -\sum_{i=1}^NX_i\dive(X_i).
\end{eqnarray*}
%

\end{proof}


\section{Invariant and totally invariant measures}

A construction of a foliated Brownian motion (FoBM) with drift can be obtained
via a Stratonovich
SDE using the gradient vector fields $X_1, \ldots  X_N$ defined before:
\begin{equation} \label{grad-FoBm}
\left\{ \begin{array}{rcl}
dX&=&V(X)~ dt+ \displaystyle \sum_{i=1}^NX_i(X)~\delta B^i\\
X_0&=&x_0 \in M,
\end{array} \right. 
\end{equation}
where  $V$ is a section of $E$ and $(B^1,\ldots,B^N)$ is the standard Brownian
motion on $\mr^N$ based on a filtered probability space $(\Omega, \mathcal{F}_t,
\mathcal{F}, \mathbf{P})$. Lemma \ref{ab} guarantees that the infinitesimal
generator
associated to the process $X_t$ is given by
$
\mathcal{L}=V+\frac{1}{2}\Delta_E.$ 

Alternative construction of a foliated Brownian motion is given via projection
on $M$ of a diffusion generated by standard vector fields in the orthonormal
frame bundle, see \cite{NA}.

Let $T_t$, with $t\geq 0$, be the Markov semigroup of operators associated to
FoBM with drift acting on $\bbl$, the space of bounded measurable
functions which are leafwise smooth.
A measure $\mu$ is invariant if
$\int_MT_tf~d\mu=\int_M f\mu$ for all $f \in \bbl$. It is equivalent to $\int
(\mathcal{L}f)~d\mu=0$.
The assumption of compactness of $M$ guarantees the existence of invariant
measures for foliated diffusions. 






A point $x$ in $ M$ is called \textit{recurrent} for the process $X$ if for all
open
neighborhoods $U$ of $x$ we have $\mP  \{\omega\in \Omega, X_{t_k}(\omega)\in
U~\textrm{ for a
sequence}~
t_k\rightarrow\infty\}=1$. A subset $U$ of $M$ is said to be \textit{saturated}
if it is the union of all the leaves
passing through points of $U$, i.e.
\[
\bigcup_{x\in U}L_x \subseteq U.
\]

 Next proposition says that the support of any
invariant measure is a saturated set.

\begin{proposition}
Let $X$ be a foliated Brownian motion with drift. The support of an invariant
measure $\mu$ is a
saturated Borel set contained in the set of recurrent points.
\end{proposition}

\begin{proof} 

Denote by $P_t(x,dy)$ the family of transition probabilities of the process. For
$x \in  \mathrm{supp} (\mu)$, by the action of $T_t^*$ in $\delta_x$ we have
that 
\[\mathrm{supp}(P_t(x,dy))\subseteq
\mathrm{supp}(T_t^*\mu)=\mathrm{supp}(\mu),
\] 
and  $ L_x \subseteq \textrm{supp}(P_t(x,~))$ for any $t>0$ since the diffusion
is
nondegenerate in the leaves. Hence
\[
 \bigcup_{x\in\textrm{supp}(\mu)}L_x\subseteq\textrm{supp} (\mu).
\]
 The result follows by the fact that supp$(\mu)$ is contained in
the closure of the subset of recurrent points of $M$ (Kliemann \cite[Lemma
4.1]{klieman}).

\end{proof}


The addition of a drift in the foliated Laplacian preserves Liouville type
theorem for harmonic functions in foliated spaces (Garnett \cite{LG}):

\begin{corollary}
 Let $X$ be a FoBM with drift and $\mu$ an invariant measure. Any function
$f\in\bbl$ satisfying $\mathcal{L}f=0$ is constant on 
every leaf $\mu$- a.s..

\end{corollary}
\begin{proof} Given such a function $f$,

\[
 \int_M||\grade f||^2 \ d\mu
= \int_M (\mathcal{L}(f^2)-2f\mathcal{L}f)~d \mu  =  0.
\]
Thus $||\grade f||=0$ on a saturated set and therefore $f$ is leafwise constant
on the $\mathrm{supp} (\mu)$.
\end{proof}

Assume that the bundle $E$ defining the foliation is oriented such that there
exists a volume form on the leaves $\chi_E\in \Omega^p(M)$  and $\upsilon \in
\Gamma(\Lambda^pE)$ with $\chi_E(\upsilon)=1$. A probability measure $\mu$ on
$M$ defines a $p$-current $\psi_\mu: \Gamma(\Lambda^pE^*) \rightarrow \mr$
which for $\alpha \in \Gamma(\Lambda^pE^*)$ is given by:
\[
 \psi_\mu(\alpha):=\int_M\alpha(\upsilon)~d\mu.
\]
A measure $\mu$ is called\textit{ totally invariant} if the associated
$p$-current
$\psi_{\mu}$ is a foliated cycle, that is   $L_X\psi_\mu=0$ for any $X\in
\Gamma(E)$ (Candel \cite{candel1}, Garnett
\cite{LG}, Sullivan
\cite{Sullivan}).
In terms of a foliated atlas, an alternatively description of a totally
invariant measure $\mu$ is via the product of the volume measure on the leaves
$\chi_E$ and a holonomy invariant measure $\nu$ in the following sense:
\[
\int f~d\mu = \sum_{\alpha \in \mathcal{U}}~\int_{S_\alpha} \left(
\int_P \lambda_\alpha f\chi_E \right) ~d\nu (P)
\]
where $\lambda_\alpha$ is a partition of unity subordinated to a foliated
atlas $\mathcal{U}= \{(U_{\alpha}, \varphi_{\alpha}): \alpha \in A \}$, $P$ are
plaques in $U_\alpha$ and $S_\alpha$
are transversal in $U_\alpha$ (see Plante \cite[p.330]{plante},  Candel
\cite[p.235]{candel}). 

The following theorem characterizes totally invariant measures in terms of
stochastic flows. 

\begin{theorem} \label{thm_carac_mti}
 A measure $\mu$ is totally invariant if and only if its associated $p$-current
$\psi_{\mu}$ is invariant by the flow of the gradient foliated Brownian motion
for each  $\omega$ a.s..
\end{theorem}
\begin{proof} Let $\phi_t$ denote the stochastic flow of the gradient FoBM with
drift of Equation (\ref{grad-FoBm}). For  any $p$-form $\alpha$, following
Kunita \cite{kunita} Theorem 4.2, we have that
\begin{eqnarray}
 \phi_t^*\alpha&=&\alpha+\int_0^t\phi_s^*L_V\alpha~ds+\sum_{i=1}
^N\int_0^t\phi_s^*L_{X_i}\alpha~ \delta B^i_s\label{formula kunita}\\
&=&\alpha+\int_0^tL_V\phi_s^*\alpha~ds+\sum_{i=1}^N\int_0^tL_{X_i}
\phi_s^*\alpha~ \hat{\delta}B^i_s\nonumber
\end{eqnarray}
where the $\hat{\delta}$ denotes the backward Stratonovich integral.

Let $\mu$ be a totally invariant measure. By definition, the integrands of
the last part of Equation ($\ref{formula kunita}$) vanishes for any $p$-form
$\alpha$. Hence
$\psi_\mu(\phi_t^*\alpha)=\psi_\mu(\alpha)$ a.s..

On the other hand, assume that for any
$p$-form $\alpha$ in $M$ we have that
$\psi_\mu(\phi_{t}^*\alpha)=\psi_\mu(\alpha)$ a.s.. Equation $(\ref{formula
kunita})$ and Doob-Meyer decomposition  implies that 
$
\psi_\mu(L_{V}\alpha)=0~\textrm{and}~
\psi_\mu(L_{X_i}\alpha)=0~\textrm{for}~ i=1,\ldots, N.$ We have to prove that
$\psi_\mu(L_X\alpha)=0$ for all $X\in \Gamma(E)$ and all $p-$form $\alpha$ in
$M$. 
Any $p$-form $\alpha$ can be written as $\alpha=f\chi_E+\beta$
with $\beta$ a $p$-form such that $\beta(\upsilon)=0$. 

We have that $\psi_\mu(L_X\beta)=0$ since
\[
 L_X\beta(\upsilon)=X(\beta(\upsilon))-\sum_{j=1}^p
g([X,v_j],v_j)~\beta(\upsilon)=0
\]
for a local expression of $\upsilon=v_1\wedge\cdots\wedge v_p$ in terms of
orthonormal sections in $\Gamma(E)$.

Let $X=\sum_{i=1}^Na_iX_i$ for some smooth functions $a_i$. We have that
\begin{eqnarray*}
\psi_\mu(L_X(f\chi_E))&=& \psi_\mu((X(f)+f\dive(X))~\chi_E)\\
&=&\sum_{i=1}^N\psi_\mu\left(a_iX_i(f)\chi_E+X_i(a_i)f\chi_E+a_ifL_{X_i}
\chi_E\right)\\
&=&\sum_{i=1}^N\psi_\mu\left(L_{X_i}(fa_i\chi_E)\right)\\
&=&0.
\end{eqnarray*}
\end{proof}

The group action of the flow $\phi_t$ in the $p$-current $\phi_{\mu}$ associated
to a measure $\mu$ is a $p$-current $\phi_{\mu_t}$ associated to a measure
$\mu_t$. In fact, direct calculation shows that $ \mu_t = \phi_{t*} \left(
\mathrm{det}_E (\phi_{t*}) \mu  \right)$,
where $\mathrm{det}_E (\phi_{t*})= \chi_E(\phi_{t*} (\upsilon))$ is the
determinant in the leaf. Hence, 
the action of the flow $\phi_t$ in the $p$-currents induces an action of the
flow in the space of measures given by $\phi_t \star \mu := \mu_t$. Denoting by
$\theta_t$ the canonical shift in the probability space $\Omega$, the cocycle
property of the flow implies that

\begin{corollary} The group action of the flow $\phi_t$ on the measures
\[
\mu_t=\phi_t \star \mu=  \phi_{t*} \left( \mathrm{det}_E (\phi_{t*}) \mu 
\right)
\]
   satisfies the cocycle property 
\[ 
\phi_s (\theta_t \omega) \star \mu_t (\omega) = \mu_{t+s}(\omega).
\]
 A deterministic measure $\mu$ is totally invariant if and only if it is a fixed
point of the action of $\phi_t$ a.s..
 
\end{corollary}
\begin{proof} The formula follows imediately from the group action of $\phi_t$
on $p$-currents. Totally invariance comes from Theorem \ref{thm_carac_mti}.
 
\end{proof}



\section{Ergodic Measures}\label{ergodic measures}

In this section we study the support of ergodic invariant probability measures
in $M$ for the foliated
Brownian motion with drift.
A \textit{minimal set} $K$ is a closed nonempty saturated set with
the property that if $K' \subseteq K$ is again a nonempty closed
saturated set, then $K=K'$.  A
\textit{transitive set} is a minimal set such that there exists at least one
dense leaf, i.e. the transitive sets are closures of the leaves.
Lemma \ref{sop. medida ergodica} below implies that the support of ergodic
measures always contains a minimal set.


%




\begin{lemma} \label{sop. medida ergodica}
 Let $\mu$ be an ergodic invariant measure for the foliated Brownian motion
with drift. The support $\textrm{supp}(\mu)$ is a transitive set. Moreover for
any minimal set $K$ we have that
$\mu (K)=0$ or $\mu(K)=1$.
\end{lemma}
\begin{proof} For $\mu$- almost
every point $x \in  \mathrm{supp} (\mu)$ 
we have the weak limit:
\begin{equation}\label{eq. medida} 
 \mu(dy)=\lim_{t\rightarrow\infty}\frac{1}{t}\int_0^tP_{s}(x,
dy)~ds,
\end{equation}
with
\begin{equation}\label{eq. soporte}
 \textrm{supp}(\mu)=\overline{ \bigcup_{t> 0}
\textrm{supp}(P_{t}(x,dy))}.
\end{equation}
 But for all $t>0 $, $\mathrm{supp}(P_t(x,dy))=\overline{L_x}$, the
leaf through $x$. 

The second statement is straighforward by invariance of $K$ and ergodicity of
$\mu$.

\end{proof}


\begin{corollary} 
 Let $K\subset M$ be a minimal set. There always exists an ergodic measure
suported on $K$.
\end{corollary}

\begin{proof} By compactness there always exists an ergodic measure  $\mu$ with
support contained in
$K$. Lemma \ref{sop. medida ergodica} implies that $\mathrm{sup}(\mu)=K$.

\end{proof}



\subsection{Application to stable foliations:} Consider the foliation of $M$
given by a strongly stable diffeomorphism $\phi:M \rightarrow M$, i.e.  
the leaf 
through a point $x$ of $M$ is given by 
\[
L_x=\{y\in M,\:d(\phi^n(x),\phi^n(y))\rightarrow
0\hspace{.3cm}\textrm{as}\hspace{.3cm}n\rightarrow \infty\}.
\]
A diffeomorphism $\phi:M\rightarrow M$ is \textit{conservative} if, for all
nonempty measurable subset $A\subset M$, 
we have that 
\[
\phi^{-j}(A)\cap\phi^{-k}(A)\neq\emptyset
\] 
for all $j,k\in\mn\cup\{0\}$. A function $f$ on $M$ which is invariant by
$\phi$, i.e. $f=f\circ \phi$ is constant in the leaves $\nu$-a.s. for any
measure $\nu$ which is $\phi$-invariant, see Y. Coudene \cite{coudene}. We have
the following criteria to $\nu$-ergodicity 
of $\phi$:

 

\begin{proposition} \label{ergodicity}
Let $\phi:M \rightarrow M$ be a strongly stable conservative transformation
which preserves a
probability  measure $\nu$. If $ \nu$ is equivalent to an ergodic harmonic
probability measure $\mu$ (w.r.t. FoBM) then $\nu$ is
$\phi$-ergodic.
\end{proposition}

\begin{proof}
Let $A\subset M$ be such that $\phi^{-1}(A) = A$, hence for $f= 1_A$ we have
that  
\[
 f^* = \lim_{t \rightarrow \infty}\frac{1}{t}\int_0^tT_sf(x)~ds=\int_M f\ d\mu \
\ \ \ \mu-a.s.
\]
On the other hand by the Coudene's result mentioned above \cite{coudene} we have
that $f$ is constant in the
leaves, hence $T_tf=f$, $\nu$-a.s. therefore $f^* =f$ $\nu$-a.s.. 

\end{proof}






\subsection{Lyapunov Exponents}

Let $\phi_t$ be the stochastic flow associated to the foliated Brownian motion
with drift of Equation (\ref{grad-FoBm}) and consider $\mu $ an ergodic
invariant probability measure in $M$.
The sum of the Lyapunov exponents $\lambda_\Sigma(x)$ at a point $x\in M$
including multiplicity is given by the limit
\[
 \lambda_\Sigma(x)=\lim_{t\rightarrow\infty}\frac{1}{t}\ln|\det(\phi_{t*}(x,))|
\]
which exists and is constant $\mP \times \mu$-almost surely for  $(\omega, x)\in
\Omega \times M$ according to (multiplicative) ergodic theorems for stochastic
flows. It\^o formulae for the logarithm of this determinant have been obtained
by various authors:

\begin{eqnarray*}
\ln(|\det(\phi_{t*}(x))|)&=&\sum_{i=1}^N\int_0^t\textrm{div}
(X_i)(\phi_s(x))~dB_s^i\\
&&+\int_0^t\left(\textrm{div}(V)+\frac{1}{2}\sum_{i=1}^NX_i\textrm{div}
(X_i)\right)(\phi_s(x))~ds.
\end{eqnarray*} 
Birkhoff's theorem in the skew-product flow (Furtenberg-Khasminskii type
argument) leads to the Baxendale's ergodic formula:
\begin{equation}\label{mean exponent}
 \lambda_\Sigma=\int_M \left(\textrm{div}(V)+\frac{1}{2}\sum_{i=1}
^NX_i\textrm{
div}(X_i)\right) d\mu\ \ \ \ \ \ \ \mP \times \mu-a.s..
\end{equation}
See e.g. Chapell \cite{chapell}, Arnold \cite{L_Arnold} and many references
therein.

We have the following expression which involves the geometry of the leaves for
these ergodic theorems.
 
\begin{theorem}\label{mean lyap exp fobm} Let $\mu$ be an ergodic probability
measure for a gradient foliated Brownian motion
with drift $V$.
Then the sum of the Lyapunov exponents is given by
\[
 \lambda_\Sigma=-\frac{1}{2}\int_M\left(||H||^2-\dive(2V-\kappa)+2g(\kappa,
V)\right)~d\mu.
\]

\end{theorem}
\begin{proof} Using the formula of Corollary \ref{formulas_H_X_i} and the
definition of the tensor $\kappa $ we have that

\begin{eqnarray*}
 \lambda_\Sigma&=&\frac{1}{2}\int_M\left(2\,
\mathrm{div}(V)+\sum_{i=1}^NX_i\dive(X_i)-
\dive(\kappa)\right)~d\mu\\
&=&\frac{1}{2}\int_M\left(\sum_{i=1}^NX_i\dive(X_i)+\dive(2V-\kappa)-2g(\kappa,
V)\right)~d\mu\\
&=&-\frac{1}{2}\int_M\left(||H||^2-\dive(2V-\kappa)+2g(\kappa,V)\right)~d\mu.\\
\end{eqnarray*}

\end{proof}

\begin{corollary} If the ergodic measure $\mu$ is harmonic totally invariant
then the sum of the Lyapunov exponents depends only on the second fundamental
form of the leaves: 
\[
 \lambda_\Sigma=-\frac{1}{2}\int_M||H||^2~d\mu.
 \]
\end{corollary}
\begin{proof} With $V=0$, use that 
$ \int_M \dive X ~d\mu =0$, for all $X \in \Gamma (E)$ see \cite[Thm 4.3]{NA}.

\end{proof}

\end{document}